\documentclass[12pt]{article}

\usepackage{amsmath,amssymb,amsbsy,amsfonts,amsthm,latexsym,
            amsopn,amstext,amsxtra,amscd,stmaryrd,fullpage}

\usepackage[mathscr]{eucal}
\usepackage{color}
\begin{document}

\newtheorem{theorem}{Theorem}
\newtheorem{lemma}[theorem]{Lemma}
\newtheorem{corollary}[theorem]{Corollary}
\newtheorem{proposition}[theorem]{Proposition}

\theoremstyle{definition}
\newtheorem*{definition}{Definition}
\newtheorem*{remark}{Remark}
\newtheorem*{example}{Example}


\def\cA{\mathcal A}
\def\cB{\mathcal B}
\def\cC{\mathcal C}
\def\cD{\mathcal D}
\def\cE{\mathcal E}
\def\cF{\mathcal F}
\def\cG{\mathcal G}
\def\cH{\mathcal H}
\def\cI{\mathcal I}
\def\cJ{\mathcal J}
\def\cK{\mathcal K}
\def\cL{\mathcal L}
\def\cM{\mathcal M}
\def\cN{\mathcal N}
\def\cO{\mathcal O}
\def\cP{\mathcal P}
\def\cQ{\mathcal Q}
\def\cR{\mathcal R}
\def\cS{\mathcal S}
\def\cU{\mathcal U}
\def\cT{\mathcal T}
\def\cV{\mathcal V}
\def\cW{\mathcal W}
\def\cX{\mathcal X}
\def\cY{\mathcal Y}
\def\cZ{\mathcal Z}


\def\sA{\mathscr A}
\def\sB{\mathscr B}
\def\sC{\mathscr C}
\def\sD{\mathscr D}
\def\sE{\mathscr E}
\def\sF{\mathscr F}
\def\sG{\mathscr G}
\def\sH{\mathscr H}
\def\sI{\mathscr I}
\def\sJ{\mathscr J}
\def\sK{\mathscr K}
\def\sL{\mathscr L}
\def\sM{\mathscr M}
\def\sN{\mathscr N}
\def\sO{\mathscr O}
\def\sP{\mathscr P}
\def\sQ{\mathscr Q}
\def\sR{\mathscr R}
\def\sS{\mathscr S}
\def\sU{\mathscr U}
\def\sT{\mathscr T}
\def\sV{\mathscr V}
\def\sW{\mathscr W}
\def\sX{\mathscr X}
\def\sY{\mathscr Y}
\def\sZ{\mathscr Z}


\def\fA{\mathfrak A}
\def\fB{\mathfrak B}
\def\fC{\mathfrak C}
\def\fD{\mathfrak D}
\def\fE{\mathfrak E}
\def\fF{\mathfrak F}
\def\fG{\mathfrak G}
\def\fH{\mathfrak H}
\def\fI{\mathfrak I}
\def\fJ{\mathfrak J}
\def\fK{\mathfrak K}
\def\fL{\mathfrak L}
\def\fM{\mathfrak M}
\def\fN{\mathfrak N}
\def\fO{\mathfrak O}
\def\fP{\mathfrak P}
\def\fQ{\mathfrak Q}
\def\fR{\mathfrak R}
\def\fS{\mathfrak S}
\def\fU{\mathfrak U}
\def\fT{\mathfrak T}
\def\fV{\mathfrak V}
\def\fW{\mathfrak W}
\def\fX{\mathfrak X}
\def\fY{\mathfrak Y}
\def\fZ{\mathfrak Z}


\def\C{{\mathbb C}}
\def\F{{\mathbb F}}
\def\K{{\mathbb K}}
\def\L{{\mathbb L}}
\def\N{{\mathbb N}}
\def\Q{{\mathbb Q}}
\def\R{{\mathbb R}}
\def\Z{{\mathbb Z}}
\def\E{{\mathbb E}}
\def\T{{\mathbb T}}
\def\P{{\mathbb P}}
\def\D{{\mathbb D}}


\def\eps{\varepsilon}
\def\mand{\qquad\mbox{and}\qquad}
\def\\{\cr}
\def\({\left(}
\def\){\right)}
\def\[{\left[}
\def\]{\right]}
\def\<{\langle}
\def\>{\rangle}
\def\fl#1{\left\lfloor#1\right\rfloor}
\def\rf#1{\left\lceil#1\right\rceil}
\def\le{\leqslant}
\def\ge{\geqslant}
\def\ds{\displaystyle}

\def\xxx{\vskip5pt\hrule\vskip5pt}
\def\yyy{\vskip5pt\hrule\vskip2pt\hrule\vskip5pt}
\def\imhere{ \xxx\centerline{\sc I'm here}\xxx }

\newcommand{\comm}[1]{\marginpar{
\vskip-\baselineskip \raggedright\footnotesize
\itshape\hrule\smallskip#1\par\smallskip\hrule}}


\def\e{\mathbf{e}}
\def\sPrc{{\displaystyle \sP_r^{(c)}}}


\title{\bf  Zeros of random linear combinations of OPUC with complex Gaussian coefficients}

\author{
{\sc Aaron M.~Yeager} \\
{Department of Mathematics, Oklahoma State University} \\
{Stillwater, OK 74078 USA} \\
{\tt aaron.yeager@okstate.edu}}

\maketitle

\begin{abstract}
We study zero distribution of random linear combinations of the form
$$P_n(z)=\sum_{j=0}^n\eta_j\phi_j(z),$$
in any Jordan region $\Omega \subset \mathbb C$. The basis functions $\phi_j$ are orthogonal polynomials on the unit circle (OPUC) that are real-valued on the real line, and $\eta_0,\dots,\eta_n$ are complex-valued iid Gaussian random variables.  We derive an explicit intensity function for the number of zeros of $P_n$ in $\Omega$ for each fixed $n$. Using the Christoffel-Darboux formula, the intensity function takes a very simple shape.  Moreover, we give the limiting value of the intensity function when the orthogonal polynomials are associated to Szeg\H{o} weights.
\end{abstract}


\section{Introduction}

The study of the expected number of real zeros of polynomials $P_n(z)=\sum_{j=0}^n\eta_jz^j$ with random coefficients, called \emph{random algebraic polynomials}, dates back to the 1930's.   In 1932, Bloch and P\'{o}lya \cite{BP} showed that when $\{\eta_j\}$ are iid random variables that take values from the set $\{-1,0,1\}$ with equal probabilities, the expected number of real zeros is $O(\sqrt{n})$.  Other early advancements in the subject were later made by Littlewood and Offord \cite{LO}, Kac \cite{K1}, \cite{K2}, Rice \cite{R}, Erd\H{o}s and Offord \cite{EO}, and many others.  For a nice history of the early progress in this topic, we refer the reader to the books by Bharucha-Reid and Sambandham \cite{BRS} and by Farahmand \cite{FB}.

There has been a lot of work done in producing formulas for the density function, called the intensity function, for the expected value of the number of zeros, whether real or complex, of $P_n(z)$.  In the 1940's Kac's \cite{K1}, \cite{K2} gave a formula for the intensity function of the expected number of real zeros of $P_n(z)$ when each $\eta_j$ are independent real Gaussian coefficients.  Using that formula he was able to show that the expected number of real roots of the random algebraic polynomial is asymptotic to $2\pi^{-1}\log n$ as $n\rightarrow \infty$.  The error term in his asymptotic was further sharpened by Hammersley \cite{HM}, Wang \cite{WG}, Edelman and Kostlan \cite{EK}, and Wilkins \cite{WL}.

Extending Kac's formula, in 1995 Shepp and Vanderbei \cite{SV} produced a formula for the intensity function in the complex plane for the expected number of zeros of $P_n(z)$ when the random variables are real-valued iid standard Gaussian. They  also obtained  the limit of the intensity function as $n\rightarrow \infty$.  Under the assumption that the random variables are from the domain of attraction of a  stable law, in 1997 their result  was generalized by Ibragimov and Zeitouni \cite{IZ}.

Some authors have  produced formulas for the intensity function of $P_n(z)$ when the random variables are complex-valued iid Gaussian.  Results in this direction can be found in the works of Farahmand \cite{F},\cite{F1}, Farahmand and Jahangiri \cite{KJ}, and Farahmand and Grigorash \cite{FG}.

Others have derived formulas for the intensity function for what are known as Gaussian Analytic Functions (GAF) $P(z)=\sum_{j=0}^{\infty}\eta_j f_j(z)$, where the $f_j$'s are square summable analytic functions on a domain, and the $\eta_j$'s are iid  Gaussian random variables, in terms of the distributional Laplacian.  For the case when the random variables are complex-valued iid Gaussian, in 2000 Hough, Krishnapur, Peres, and Vir$\acute{\text{a}}$g (Section 2.4.2, pp. 24-29 of \cite{ZGAF}) and Feldheim (Theorem 2, p. 6 of \cite{FL}) derived the intensity function of zeros. Feldheim obtained the intensity function of zeros for a GAF in the same paper (Theorem 3, p. 7 of \cite{FL}) when the $f_j$'s are real-valued on the real line and the $\eta_j$'s are real-valued iid standard Gaussian random variables.

In 2015, Vanderbei \cite{CZRS} produced an explicit formula of the intensity function for  finite sums of real-valued iid standard Gaussian random variables with the spanning functions taken to be entire functions that are real-valued on the real line.  Vanderbei also gave the limiting intensity function when the spanning functions are Weyl polynomials, Taylor polynomials, and the truncated Fourier series.

Using Vanderbei's method, in a recent paper the author \cite{AY} derived an explicit formula for the expectation of the number of zeros of a random sum in a Jordan domain when the random variables are complex-valued iid Gaussian and the spanning functions are entire functions that are real-valued on the real line.  This formula for the intensity function was presented by the author and on the same day presented in a slightly different form by Andrew Ledoan \cite{AL} at the $15^{\text{th}}$ International Conference in Approximation Theory in San Antonio, TX, on May $25^{\text{th}}$.  As an application of the author's formula, in \cite{AY} the spanning functions were taken to be orthogonal polynomials on the real line (OPRL).  Using the Christoffel-Darboux formula to simplify the intensity function, the shape of the intensity function became so manageable that when the OPRL were associated to a general weight class, using classical results in \cite{SZ}, and modifications of them, the limiting value of the intensity function was obtained.

For results concerning the real zeros of random orthogonal polynomials with the random coefficients being real-valued iid standard Gaussian, we refer the reader to the works of Das \cite{D}, Das and Bhatt \cite{DB}, Lubinsky, Pritsker, and Xie \cite{LPX}, \cite{LPX2}, and Pritsker and Xie \cite{PX}.

In this paper, we give the analogues of the results from \cite{AY} concerning OPRL for orthogonal polynomials on the unit circle (OPUC). The \emph{orthogonal polynomials on the unit circle associated to a weight $W(\theta)$}, where $W(\theta)$ is a non-negative $2\pi$ periodic function that is Lebesgue integrable on $[-\pi,\pi]$ such that
$$\int_{-\pi}^{\pi} W(\theta)\  d\theta >0,$$
are polynomials $\{\phi_n(z)\}$ that satisfy
$$\frac{1}{2\pi}\int_{-\pi}^{\pi}W(\theta) \phi_n(e^{i\theta})\overline{\phi_m(e^{i\theta})}\ d\theta =\delta_{nm},$$
for all $n,m\in \N\cup \{0\}$. It is known that each $\phi_n(z)$ is a polynomial of exact degree $n$, and that the leading coefficient of $\phi_n(z)$, denoted as $\kappa_n$, is real and positive.  Given these properties, the polynomials $\{\phi_n(z)\}$ are uniquely determined.

Many examples and properties of OPUC are explored in the books by Szeg\H{o} \cite{SZ} and Simon \cite{BS}.  One example of an OPUC that we have already mentioned are the monomials, that is $z^n$ for $n\in \N\cup\{0\}$.  Hence for case when the random variables are complex-valued iid Gaussian and the spanning functions are the monomials, the works of  Farahmand \cite{F}, \cite{F1} and  Farahmand and Jahangiri \cite{KJ}, provide formulas for the intensity function and its limiting value.  The paper of Shiffman and Zelditch \cite{SHZ2} mentions a heuristic argument that provides the intensity function and its asymptotic for random polynomials spanned by OPUC associated to analytic weights in terms of the distributional Laplacian.

Other authors have studied the asymptotic zero distribution for random polynomials spanned by orthogonal polynomials with respect to various measures.   There has also been work done in the higher dimensional analogs of these settings, see Shiffman and Zelditch \cite{SHZ1}-\cite{SHZ3}, Bloom \cite{BL1} and \cite{BL2}, Bloom and Shiffman \cite{BLSH}, Bloom and Levenberg \cite{BLL}, Bayraktar \cite{BY}, and Pritsker \cite{P1}, \cite{P2}.

To prove our first theorem, we will use Theorem 1 from \cite{AY}. For convenience of the reader, we reformulate the statement of this theorem.

Let $\{f_j(z)\}_{j=0}^n$ be a sequence of entire functions in the complex plane that are real-valued on the real line.  Let
\begin{equation}\label{P}
P_n(z)=\sum_{j=0}^n \eta_j f_j(z), \ \ \ \ z\in\C,
\end{equation}
where $n$ is a fixed integer, and $\eta_j=\alpha_j + i \beta_j$, $j=0,1, \dots, n$, with $\{\alpha_j\}_{j=0}^n$ and $\{\beta_j\}_{j=0}^n$ being sequences of independent standard normal random variables. Let $N_n(\Omega)$ denote the (random) number of zeros of $P_n$ in a Jordan region $\Omega$ of the complex plane, and let the mathematical expectation be denoted by $\E$.
The formula we need is expressed in terms of the kernels
\begin{equation}\label{KN0}
K_{n}(z,w)=\sum_{j=0}^{n}f_j(z)\overline{f_j(w)},\ \ \ \ \ \ \ K_{n}^{(0,1)}(z,w)=\sum_{j=0}^{n} f_j(z)\overline{f_j^{\prime}(w)},
\end{equation}
and
\begin{equation}\label{KN2}
K_{n}^{(1,1)}(z,w)=\sum_{j=0}^{n}f_j^{\prime}(z)\overline{f_j^{\prime}(w)}.
\end{equation}

The mentioned theorem in \cite{AY} states that
for each Jordan region $\Omega \subset \{z\in \C : K_{n}(z,z)\neq 0\}$, the intensity function $h_n(z)$ for the random sum \eqref{P} with complex-valued iid Gaussian coefficients, is given by
$$\E[N_n(\Omega)]=\int_{\Omega}h_n(x,y)\ dx \ dy ,$$
with
\begin{equation}\label{thm2.1}
h_n(x,y)=h_n(z)=\frac{K_{n}^{(1,1)}(z,z)K_{n}(z,z)-\left|K_{n}^{(0,1)}(z,z)\right|^2}{\pi \left(K_{n}(z,z)\right)^2},
\end{equation}
where the kernels $K_n(z,z)$, $K_n^{(0,1)}(z,z)$, and $K_n^{(1,1)}(z,z)$, are defined  in \eqref{KN0} and \eqref{KN2}.

We set $f_j(z)=\phi_j(z)$ in \eqref{P}, where $\phi_j(z)$ are OPUC and study the expected number of zeros of
\begin{equation}\label{PnOPUC}
P_n(z)=\sum_{j=0}^n \eta_j \phi_j(z)
\end{equation}
in a Jordan region $\Omega$, where $n$ is a fixed integer, and $\eta_j=\alpha_j + i \beta_j$, $j=0,1, \dots, n$, with $\{\alpha_j\}_{j=0}^n$ and $\{\beta_j\}_{j=0}^n$ being sequences of independent standard normal random variables.

We note that when the weight function $W(\theta)$ associated to the OPUC $\{\phi_j(z)\}$ is an even function, the coefficients of each $\phi_j(z)$, $j=0,1,\dots$, are real.

Appealing to the Christoffel-Darboux formula for OPUC to simplify the kernels $K_n(z,z)$, $K_n^{(0,1)}(z,z)$, and $K_n^{(1,1)}(z,z)$ as defined by \eqref{KN0} and \eqref{KN2}, and applying \eqref{thm2.1} we obtain the following:
\begin{theorem}\label{OPUC}
Let $W(\theta)$ be an even weight function associated to the OPUC $\{\phi_j(z)\}$.  When $|z|\neq 1$, the intensity function $h_n^P(z)$ for the random orthogonal polynomial \eqref{PnOPUC} spanned by the $\phi_j$'s with complex-valued iid Gaussian coefficients, simplifies to
\begin{align*}
h_n^P(z)
&=\frac{1}{\pi\left(1-|z|^2 \right)^2}\Bigg(1-\frac{\left(1-|z|^2 \right)^2\left|\phi_{n+1}^{*}(z)\phi_{n+1}^{\prime}(z)-\phi_{n+1}^{*\ \prime}(z)\phi_{n+1}(z)  \right|^2}{\left(\left|\phi_{n+1}^{*}(z)\right|^2-\left|\phi_{n+1}(z)\right|^2\right)^2}   \Bigg),
\end{align*}
where $\phi_n^{*}(z)=z^n \overline{\phi_n\left(\frac{1}{\bar{z}}\right)}$.
\end{theorem}

We conclude the paper by giving the limiting value of the intensity function for \eqref{PnOPUC} when the weight function is from the Szeg\H{o} weight class.  We say that $W(\theta)\geq 0$ belongs to the \emph{Szeg\H{o} weight class}, denoted by $G$, if $W(\theta)$ is defined and measurable in $[-\pi,\pi]$, and the integrals
$$\int_{-\pi}^{\pi} W(\theta)\ d\theta, \ \ \ \ \int_{-\pi}^{\pi}|\log W(\theta)|\ d\theta$$
exist with the first integral assumed to be positive.  Taking $W(\theta)\in G$ and using limits and asymptotics by Szeg\H{o} in \cite{SZ} (Theorems 11.3.3 on p. 291 and 12.1.1 on p. 297) and modifications of them, we are able to achieve:

\begin{theorem}\label{OPUC1}
Let $W(\theta)\in G$ be an even weight function for the OPUC $\{\phi_j(z)\}$.   Then the intensity function $h_n^P(z)$ for the random orthogonal polynomial \eqref{PnOPUC} spanned by the $\phi_j$'s with complex-valued iid Gaussian coefficients satisfies
\begin{equation}\label{LIOPUC}
\lim_{n\rightarrow \infty}h_n^P(z)= \frac{1}{\pi\left(1-|z|^2 \right)^2}
\end{equation}
for $|z|\neq1$.  Moreover the convergence in \eqref{LIOPUC} holds uniformly on compact subsets of $\{z:|z|\neq 1\}$.
\end{theorem}

To prove the above Theorem we consider the separate cases of when $|z|<1$ and $|z|>1$.  When $|z|<1$ we give two different proofs of the limiting value of $h_n^P$.  Both of these proofs are given since they contain asymptotics we derive that might be useful in other applications.  We prove the case when $|z|>1$ by establishing limits and asymptotics, some of which are in the literature and some are not, for all the functions that make up intensity function $h_n^P$ in Theorem \ref{OPUC}.

\section{Derivation of the Intensity Function}

Since the polynomials $\phi_j(z)$, $j=0,1,\dots,n$,  are orthogonal on the unit circle, we can simplify the kernels $K_{n}(z,z)$, $K_{n}^{(0,1)}(z,z)$, and $K_{n}^{(1,1)}(z,z)$ which make up the intensity function $h_n^{P}$ using the Christoffel-Darboux formula for OPUC.  For convenience of the reader, the Christoffel-Darboux formula for OPUC  (Theorem 2.2.7, p. 124 of \cite{BS}) states that for $z,w\in \C$ and $\{\phi_j(z)\}_{j\geq 0}$ OPUC, we have
\begin{equation}\label{CDUP}
\sum_{j=0}^{n}\phi_j(z)\overline{\phi_j(w)}= \frac{\overline{\phi_{n+1}^{*}(w)}\phi_{n+1}^{*}(z)-\overline{\phi_{n+1}(w)}\phi_{n+1}(z)}{1-\bar{w}z},
\end{equation}
where $\phi_n^{*}(z)=z^n \overline{\phi_n\left(\frac{1}{\bar{z}}\right)}$.

Before obtaining our representations of the kernels, let us note that since we are assuming that the weight function $W(\theta)$ is an even function, the polynomials $\phi_j(z)$, $j=0,1,\dots$, have real coefficients.  Thus when using conjugation we have that $\overline{\phi_j(z)}=\phi_j(\bar{z})$ for all $j=0,1,\dots$, and all $z\in \C$.

\begin{proof}[Proof of Theorem \ref{OPUC}]
Using the Christoffel-Darboux Formula to get a representation for $K_{n}(z,z)$, we take $w=z$ in \eqref{CDUP} to achieve
\begin{align}\label{CDBOUP}
K_{n}(z,z)&=\sum_{j=0}^{n}\phi_j(z)\overline{\phi_j(z)} =\frac{\left|\phi_{n+1}^{*}(z)\right|^2-\left|\phi_{n+1}(z)\right|^2}{1-\left|z\right|^2}.
\end{align}

For our representation of $K_{n}^{(0,1)}(z,z)$, we first take the derivative of \eqref{CDUP} with respect to $\bar{w}$.  Taking this derivative and writing $\phi_{n+1}^{*\ \prime}(\bar{w})$ to mean $\frac{d}{d\bar{w}}[\phi_{n+1}^*(\bar{w})]$, it follows that
\begin{align}
\nonumber
\sum_{j=0}^{n}\phi_j(z)\overline{\phi_j^{\prime}(w)}&= \frac{\phi_{n+1}^{*\ \prime}(\bar{w})\phi_{n+1}^{*}(z)-\phi_{n+1}^{\prime}(\bar{w})\phi_{n+1}(z)}{1-\bar{w}z}\\
\label{DWRTW}
&\ \ \  +\frac{z\left(\overline{\phi_{n+1}^{*}(w)}\phi_{n+1}^{*}(z)-\overline{\phi_{n+1}(w)}\phi_{n+1}(z)\right)}{(1-\bar{w}z)^2}.
\end{align}
Setting $w=z$ in the above and recalling \eqref{CDBOUP} we have
\begin{align}\label{CDB1UP}
K_n^{(0,1)}(z,z)&=\sum_{j=0}^{n}\phi_j(z)\overline{\phi_j^{\prime}(z)}=\frac{\overline{\phi_{n+1}^{*\ \prime}(z)}\phi_{n+1}^{*}(z)-\overline{\phi_{n+1}^{\prime}(z)}\phi_{n+1}(z)}{1-|z|^2}
+\frac{zK_n(z,z)}{1-|z|^2}.
\end{align}

To obtain a representation for $K_n^{(1,1)}(z,z)$ we differentiate \eqref{DWRTW} with respect to $z$, then setting $w=z$ and using \eqref{CDBOUP} and \eqref{CDB1UP} gives
\begin{align}\label{CDB2UP}
K_n^{(1,1)}(z,z)&=\sum_{j=0}^n|\phi_j^{\prime}(z)|^2
=\frac{|\phi_{n+1}^{* \ \prime}(z)|^2-|\phi_{n+1}^{\prime}(z)|^2}{1-|z|^2}+ \frac{\bar{z}K_n^{(0,1)}(z,z)
 +z\overline{ K_n^{(0,1)}(z,z)}+K_n(z,z)}{1-|z|^2}.
\end{align}

Using \eqref{CDBOUP}, \eqref{CDB1UP}, and \eqref{CDB2UP}, the numerator of the intensity function from \eqref{thm2.1} simplifies as
\begin{align*}
K_n^{(1,1)}(z,z)K_n(z,z)-&\left|K_n^{(0,1)}(z,z)\right|^2  \\
&=\frac{\left(K_n(z,z)\right)^2}{\left(1-|z|^2\right)^2}-\frac{\left|\phi_{n+1}^*(z)\phi_{n+1}^{\prime}(z)-\phi_{n+1}^{*\ \prime}(z)\phi_{n+1}(z)  \right|^2  }{ \left(1-|z|^2\right)^2  }.
\end{align*}

Therefore, using the above numerator and \eqref{CDBOUP}, the intensity function from \eqref{thm2.1} becomes
\begin{align*}
h_n^P(z)&=\frac{K_{n}^{(1,1)}(z,z)K_{n}(z,z)-\left|K_{n}^{(0,1)}(z,z)\right|^2}{\pi \left(K_{n}(z,z)\right)^2} \\
&=\frac{1}{\pi\left(1-|z|^2 \right)^2}\Bigg(1-\frac{\left(1-|z|^2 \right)^2\left|\phi_{n+1}^{*}(z)\phi_{n+1}^{\prime}(z)-\phi_{n+1}^{*\ \prime}(z)\phi_{n+1}(z)  \right|^2}{\left(\left|\phi_{n+1}^{*}(z)\right|^2-\left|\phi_{n+1}(z)\right|^2\right)^2}   \Bigg)
\end{align*}
which gives the result of Theorem \ref{OPUC}.
\end{proof}

\section{The Limiting Value of the Intensity Function}

As mentioned in the introduction we will prove Theorem \ref{OPUC1} in separate cases when $|z|<1$ and $|z|>1$. We begin by proving the case for when $|z|<1$ and give both methods of proof for this case.

\begin{proof}[Proof of Theorem \ref{OPUC1}]
By equation (12.3.17) on page 303 in \cite{SZ}, when $|z|<1$ and $|w|<1$ we have
\begin{equation}\label{KNL}
\lim_{n\rightarrow \infty}K_n(z,w)=\sum_{j=0}^{\infty}\phi_j(z)\overline{\phi_j(w)}=\frac{1}{(1-\overline{w}z)\overline{D(w)}D(z)},
\end{equation}
where
$$D(\xi)=\exp\left\{ \frac{1}{4\pi}\int_{-\pi}^{\pi}\log W(t) \ \frac{1+\xi e^{-it}}{1-\xi e^{-it}}\ dt \right\},$$
is a uniquely determined function from the non-negative $2\pi$-periodic weight function $W(\theta)$, with
$$\int_{-\pi}^{\pi} W(\theta)\ d\theta >0,$$
that is analytic and nonzero for $|\xi|<1$ with $D(0)>0$.  When $|z|<1$, for the first method of proof  our main tool for expressing the limits of the kernels will be \eqref{KNL}.  We note that to express the limits of the kernels $K_n^{(0,1)}(z,z)$ and $K_n^{(1,1)}(z,z)$ we will have to take derivatives of \eqref{KNL}.  Taking these derivatives is justified since \eqref{KNL} is holomorphic in $z$ and anti-holomorphic in $\overline{w}$, and because the infinite series converges for all $|z|<1$ and $|w|<1$, hence uniformly on compact sets of this domain. Thus  taking derivatives we retain the convergence of the infinite series concerning the derivatives.

With calculations analogous to the derivation of intensity function, using \eqref{KNL} to derive the limiting values of our kernels, we obtain
\begin{align}\label{KNLB0}
\lim_{n\rightarrow \infty}K_n(z,z)&=\sum_{j=0}^{\infty}|\phi_j(z)|^2=\frac{1}{(1-|z|^2)|D(z)|^2}:=K(z,z),\\
\label{KNLB1}
\lim_{n\rightarrow \infty}K_n^{(0,1)}(z,z)&=\sum_{j=0}^{\infty}\phi_j(z)\overline{\phi_j^{\prime}(z)} =\frac{zK(z,z)}{1-|z|^2}-\frac{\overline{D^{\prime}(z)}K(z,z)}{\overline{D(z)}}:=K^{(0,1)}(z,z),\\
\nonumber
\text{and}&\\
\nonumber
\lim_{n\rightarrow \infty}K_n^{(1,1)}(z,z)&=\sum_{j=0}^{\infty}\left|\phi_j^{\prime}(z)\right|^2 \\
\nonumber
&=\frac{K(z,z)+\bar{z}K^{(0,1)}(z,z)+z\overline{K^{(0,1)}(z,z)}}{1-|z|^2}+\frac{|D^{\prime}(z)|^2}{(1-|z|^2)|D(z)|^4}\\
\label{KNLB2}
&:=K^{(1,1)}(z,z).
\end{align}

Taking the limit of the numerator of the intensity function in \eqref{thm2.1}, and then using \eqref{KNLB0}, \eqref{KNLB1}, and \eqref{KNLB2},  we see that
\begin{align*}
\lim_{n\rightarrow \infty}K_{n}^{(1,1)}(z,z)K_{n}(z,z)-\left|K_{n}^{(0,1)}(z,z)\right|^2&=K^{(1,1)}(z,z)K(z,z)-\left|K^{(0,1)}(z,z)\right|^2 \\
&=
\frac{\left(K(z,z)\right)^2}{(1-|z|^2)^2}.
\end{align*}

Therefore passing to the limit of the intensity function in \eqref{thm2.1}, and using the above limit of the numerator and \eqref{KNLB0}, yields
\begin{align*}
\lim_{n\rightarrow \infty}h_n(z)&=\lim_{n\rightarrow \infty} \frac{K_{n}^{(1,1)}(z,z)K_{n}(z,z)-\left|K_{n}^{(0,1)}(z,z)\right|^2}{\pi \left(K_{n}(z,z)\right)^2}
=\frac{1}{\pi\left(1-|z|^2\right)^2},
\end{align*}
and hence proves Theorem \ref{OPUC1} in the case when $|z|<1$.

The other way to prove Theorem \ref{OPUC1} when $|z|<1$ is to take the limit as $n\rightarrow \infty$ of the intensity function from Theorem \ref{OPUC}.  Using \eqref{KNLB0} and \eqref{KNLB2} from the previous approach we have
\begin{equation}\label{phizero}
\lim_{n\rightarrow \infty}\phi_{n+1}(z)=0,
\end{equation}
and
\begin{equation}\label{phiprzero}
\lim_{n\rightarrow \infty}\phi_{n+1}^{\prime}(z)=0,
\end{equation}
for $|z|<1$.

Thus to complete the proof in this approach we need to know what the limits of $\phi_{n+1}^{*}(z)$ and $\phi_{n+1}^{*\ \prime}(z)$ are as $n\rightarrow \infty$. Since $|z|<1$, and the weight function $W(\theta)$ is an even function, which implies $\bar{D}(z)=\overline{D(\bar{z})}=D(z)$, by equation (12.3.16) on page 302 of \cite{SZ} we have
\begin{equation}\label{phistar}
\lim_{n\rightarrow \infty}\phi_{n+1}^*(z)=\left(D(z)\right)^{-1}.
\end{equation}

To find the limiting behavior of $\phi_{n+1}^{*\ \prime}(z)$, as done on page 302 of \cite{SZ}, we start by considering the function
\begin{equation*}
\phi_n^{*}(z)D(z)-1=[D(0)\kappa_n-1]+d_{n1}z+d_{n2}z^2+\cdots:=e_n(z),
\end{equation*}
which is holomorphic for $|z|<1$, and $e_n(z) \rightarrow 0$ uniformly on compact subsets of the unit disk $\mathbb{D}$.  Differentiating both sides of the above we see that
\begin{equation}\label{phistarpr0}
\phi_n^{*\ \prime}(z)D(z)+\phi_n^{*}(z)D^{\ \prime}(z)=d_{n1}+2d_{n2}z+\cdots=e_n^{\prime}(z).
\end{equation}
There are two ways to see that $e_n^{\prime}(z)\rightarrow 0$ uniformly on compact subsets of $\mathbb{D}$.  The first way is observe that since $e_n(z) \rightarrow 0$ uniformly on compact subsets of $\mathbb{D}$, and $e_n(z)$ is holomorphic for each $n$ and all $z\in \mathbb{D}$, by a standard result in complex analysis we have $e_n^{\prime}(z)\rightarrow 0$ uniformly on compact subsets of $\mathbb{D}$.  The second way to see that $e_n^{\prime}(z)\rightarrow 0$ uniformly on compact subsets of $\mathbb{D}$ is to see what the behavior of $e_n^{\prime}(z)$ would be from the proof given on page 302 in \cite{SZ}.  To this end, by Cauchy's inequality and since $|z|<1$ we have
\begin{align}
\nonumber
|d_{n1}+2d_{n2}z+\cdots|^2
&\leq (|d_{n1}|^2+|d_{n2}|^2+\cdots)\sum_{j=1}^{\infty}j^2|z|^{2j-2} \\
\nonumber
&= (|d_{n1}|^2+|d_{n2}|^2+\cdots)\frac{|z|^2+1}{(1-|z|^2)^3} \\
\label{e1}
&\leq \left[2-2(\mathfrak{G}(W))^{\frac{1}{2}}(\mu_n(W))^{-\frac{1}{2}}\right]\frac{|z|^2+1}{(1-|z|^2)^3},
\end{align}
where we are using equation (12.3.13) on page 302 of \cite{SZ} for last inequality, and $\mathfrak{G}(W)$ is the geometric mean of $W(\theta)$, and $\mu_n(W)$ is the minimum of
$$\frac{1}{2\pi}\int_{-\pi}^{\pi}W(\theta)|\rho(e^{i\theta})|^2 d\theta,$$
with $\rho(z)=z^n+a_1z^{n-1}+\cdots +a_n$ being an arbitrary polynomial of degree $n$.  Since $W$ is in the weight class $G$, by  equation (12.3.3) on page 300 of \cite{SZ} we know $\mu_n(W)\rightarrow G(W)$.  Thus \eqref{e1} tends to zero uniformly for $|z|<1$.

Since now we know that $e_n^{\prime}(z)=o(1)$, solving for $\phi_n^{* \ \prime}(z)$ in \eqref{phistarpr0}, taking the limit as $n\rightarrow \infty$ and also using \eqref{phistar}, it follows that
\begin{equation}\label{phistarpr}
\lim_{n\rightarrow \infty}\phi_n^{*\ \prime}(z)=\frac{-D^{\ \prime}(z)}{\left(D(z)\right)^2}.
\end{equation}

Thus from \eqref{phizero}, \eqref{phiprzero}, \eqref{phistar}, and \eqref{phistarpr}, we obtain
\begin{align*}
\lim_{n\rightarrow \infty}\frac{\left|\phi_{n+1}^{*}(z)\phi_{n+1}^{\prime}(z)-\phi_{n+1}^{*\ \prime}(z)\phi_{n+1}(z)  \right|^2}{\left|\phi_{n+1}^{*}(z)\right|^2-\left|\phi_{n+1}(z)\right|^2} =  0.
\end{align*}

Therefore from the above limit and using Theorem \ref{OPUC} we have
\begin{align*}
\lim_{n\rightarrow \infty}h_n^P(z)&=\lim_{n\rightarrow \infty}\frac{1}{\pi\left(1-|z|^2 \right)^2}\left(1-\frac{\left(1-|z|^2\right)^2\left|\phi_{n+1}^{*}(z)\phi_{n+1}^{\prime}(z)-\phi_{n+1}^{*\ \prime}(z)\phi_{n+1}(z)  \right|^2}{\left|\phi_{n+1}^{*}(z)\right|^2-\left|\phi_{n+1}(z)\right|^2}   \right) \\
&=\frac{1}{\pi\left(1-|z|^2 \right)^2},
\end{align*}
which completes the other way to prove Theorem \ref{OPUC1} when $|z|<1$.

We will now prove Theorem \ref{OPUC1} when $|z|>1$.  Recall that since the weight function $W(\theta)$ is an even function, it follows that the functions $\phi_{n+1}$ have real coefficients.  Thus $\phi_{n+1}^{*}(z)=z^{n+1}\overline{\phi_n(\bar{z}^{-1})}=z^{n+1}\phi_{n+1}(z^{-1})$. Differentiating  this equation gives us
\begin{equation}\label{dphistar}
\phi_{n+1}^{*\ \prime}(z)=(n+1)z^{n}\phi_{n+1}(z^{-1})-z^{n-1}\phi_{n+1}^{\prime}(z^{-1}).
\end{equation}

Note that by $|z|>1$, we have $|z^{-1}|<1$.  Thus using \eqref{KNLB0} and \eqref{KNLB2} again yields
\begin{equation}\label{phizero2}
\lim_{n\rightarrow \infty}\phi_{n+1}(z^{-1})=0, \ \ \ \text{and}\ \ \ \lim_{n\rightarrow \infty}\phi_{n+1}^{\prime}(z^{-1})=0.
\end{equation}

Since $W(\theta)$ is an even function, so that $\bar{D}(z^{-1}))=\overline{D(1/\bar{z})}=D(z^{-1})$, which is also in the Szeg\H{o} weight class, by Theorem 12.1.1 in \cite{SZ} on page 297
we have
\begin{equation}\label{phiasy1}
\lim_{n\rightarrow \infty}\frac{\phi_{n+1}(z)D(z^{-1})}{z^{n+1}}=1,
\end{equation}
and the convergence holds uniformly on $\{z:|z|\geq R>1\}$.
From the above limit we achieve
$$\phi_{n+1}(z)D(z^{-1})z^{-(n+1)}-1=e_n(z),$$
where $e_n(z)\rightarrow 0$ as $n\rightarrow \infty$ uniformly on compact subsets of $\{z:|z|\geq r>1\}$.  Since $D(z^{-1})$ is holomorphic on $\C\setminus \overline{\D}$, taking the derivative of the above yields
\begin{equation}\label{dphiasy0}
\frac{\phi_{n+1}^{\prime}(z)D(z^{-1})}{z^{n+1}}+\frac{\phi_{n+1}(z)\frac{d}{dz}D(z^{-1})}{z^{n+1}}
-\frac{(n+1)\phi_{n+1}(z)D(z^{-1})}{z^{n+2}}=e_n^{\prime}(z).
\end{equation}

We seek to show that $\lim_{n\rightarrow \infty}e_{n}^{\ \prime}(z)=0$ uniformly on compact subsets of $\{z:|z|\geq r >1\}$.
Observe that
\begin{align*}
\lim_{z \rightarrow \infty}e_n(z)&
=\lim_{z \rightarrow \infty}\left[\left(\kappa_{n+1}z^{n+1}+\kappa_{n}z^{n}+\cdots +\kappa_1z+\kappa_0\right)D(z^{-1})z^{-(n+1)}-1\right] \\
&=  \left\{
        \begin{array}{ll}
            \kappa_1D(0)-1, & \quad n = 0, \\
            \kappa_{n+1}D(0)-1, & \quad n > 0.
        \end{array}
    \right.
\end{align*}
Since both $\kappa_1D(0)-1$ and $\kappa_{n+1}D(0)-1$ exits and are finite, we have that the $\lim_{z\rightarrow \infty}e_n(z)$ exits for all $n=0,1\dots$. Consequently $e_n(z)$ is holomorphic at infinity for all $n\in \N \cup \{0\}$.

From the previous argument we see that $e_n(z)$ is holomorphic on $\overline{\C}\setminus \overline{\D}$. Thus we can apply Cauchy's Formula for unbounded domains to $e_n(z)$ on the unbounded set $\{z:|z|\geq r >1\}$. Let $K\subset \{z:|z|\geq r >1\}$ be a compact set with $w\in K$.  By Cauchy's Formula for unbounded domains we have
$$e_n(w)=\frac{-1}{2\pi i}\int_{\partial D(0,r)}\frac{e_n(t)}{t-w} \ dt +\lim_{w \rightarrow \infty} e_n(w).$$
Taking the derivative of the above with respect to $w$ and estimating yields
\begin{align*}
|e_n^{\prime}(w)|&=\left| \frac{-1}{2\pi i}\int_{\partial D(0,r)}\frac{e_n(t)}{(t-w)^2} \ dt \right|\\
&\leq \frac{1}{2\pi }\int_{\partial D(0,r)}\left|\frac{e_n(t)}{(t-w)^2}\right| \ d|t| \\
&\leq\frac{r}{(\text{Dist}(K,\partial D(0,r)))^2}\ \displaystyle\max_{t\in \partial D(0,r)}|e_n(t)|\rightarrow 0,
\end{align*}
uniformly given that $e_n(w)=o(1)$ uniformly as $n\rightarrow \infty$ for $w\in\{z:|z|\geq r >1\}$.

Therefore by $K$ being an arbitrary compact subset of $\{z:|z|\geq r >1\}$, it follows that $\lim_{n\rightarrow \infty}e_n^{\ \prime}(w)=0$ uniformly on compact subsets of $\{z:|z|\geq r >1\}$.

Using that $e_n^{\prime}(z)=o(1)$, solving for $\phi_{n+1}^{\prime}(z)$ in \eqref{dphiasy0} yields
\begin{equation}\label{dphiasy1}
\phi_{n+1}^{\prime}(z)=-\frac{\phi_{n+1}(z)\frac{d}{dz}D(z^{-1})}{D(z^{-1})}+\frac{(n+1)\phi_{n+1}(z)}{z}+o(z^n).
\end{equation}

We now have all the limits and asymptotics to start the calculation that will show
$$\lim_{n\rightarrow \infty}\frac{\left|\phi_{n+1}^{*}(z)\phi_{n+1}^{\prime}(z)-\phi_{n+1}^{*\ \prime}(z)\phi_{n+1}(z)  \right|^2}{\left(\left|\phi_{n+1}^{*}(z)\right|^2-\left|\phi_{n+1}(z)\right|^2\right)^2}=0.$$
Let us begin by observing that
\begin{align}
\nonumber
&\frac{\left|\phi_{n+1}^{*}(z)\phi_{n+1}^{\prime}(z)-\phi_{n+1}^{*\ \prime}(z)\phi_{n+1}(z)  \right|^2}{\left(\left|\phi_{n+1}^{*}(z)\right|^2-\left|\phi_{n+1}(z)\right|^2\right)^2}
=\frac{\left|\phi_{n+1}^{*}(z)\phi_{n+1}^{\prime}(z)-\phi_{n+1}^{*\ \prime}(z)\phi_{n+1}(z)  \right|^2}{\left(\left|z^{n+1}\phi_{n+1}(z^{-1})\right|^2-\left|\phi_{n+1}(z)\right|^2\right)^2 }\\
\label{tgt}
&\ \ \ \ \ \ \ \ \ \ \ \ \ \ \ \ \ =\frac{\frac{1}{|z^{2(n+1)}|^2}\left|\phi_{n+1}^{*}(z)\phi_{n+1}^{\prime}(z)-\phi_{n+1}^{*\ \prime}(z)\phi_{n+1}(z)  \right|^2}{\left(\left|\phi_{n+1}(z^{-1})\right|^2-\left|\frac{\phi_{n+1}(z)}{z^{n+1}}\right|^2\right)^2}
\end{align}
By the first limit in \eqref{phizero2} and the limit \eqref{phiasy1}, for the denominator above we have
\begin{equation}\label{tgtdm}
\lim_{n\rightarrow \infty}\left(\left|\phi_{n+1}(z^{-1})\right|^2-\left|\frac{\phi_{n+1}(z)}{z^{n+1}}\right|^2\right)^2=\frac{1}{\left|D(z^{-1})\right|^4}<\infty.
\end{equation}

Turning now to the numerator in \eqref{tgt}, using \eqref{dphistar} then \eqref{dphiasy1} we see that
\begin{align}
\nonumber
&\frac{1}{|z^{2(n+1)}|^2}\left|\phi_{n+1}^{*}(z)\phi_{n+1}^{\prime}(z)-\phi_{n+1}^{*\ \prime}(z)\phi_{n+1}(z)  \right|^2\\
\nonumber
&=\frac{1}{|z^{2(n+1)}|^2}\Big|z^{n+1}\phi_{n+1}(z^{-1})\phi_{n+1}^{\prime}(z)-(n+1)z^n\phi_{n+1}(z^{-1})\phi_{n+1}(z)  +z^{n-1}\phi_{n+1}^{\prime}(z^{-1})\phi_{n+1}(z)  \Big|^2\\
\nonumber
&=\Big|\frac{\phi_{n+1}(z^{-1})\phi_{n+1}^{\prime}(z)}{z^{n+1}}-\frac{(n+1)\phi_{n+1}(z^{-1})\phi_{n+1}(z)}{z^{n+2}}  +\frac{\phi_{n+1}^{\prime}(z^{-1})\phi_{n+1}(z)}{z^{n+3}}  \Big|^2\\
\nonumber
&=\Big|-\frac{\phi_{n+1}(z^{-1})\phi_{n+1}(z)\frac{d}{dz}D(z^{-1})}{z^{n+1}D(z^{-1})}
+\frac{(n+1)\phi_{n+1}(z^{-1})\phi_{n+1}(z)}{z^{n+2}}\\
\nonumber
&\ \ \ \ +\frac{\phi_{n+1}(z^{-1})o(z^n)}{z^{n+1}}-\frac{(n+1)\phi_{n+1}(z^{-1})\phi_{n+1}(z)}{z^{n+2}} +\frac{\phi_{n+1}^{\prime}(z^{-1})\phi_{n+1}(z)}{z^{n+3}}  \Big|^2\\
\label{tgtnum}
&=\Big|-\frac{\phi_{n+1}(z^{-1})\phi_{n+1}(z)\frac{d}{dz}D(z^{-1})}{z^{n+1}D(z^{-1})}
+\frac{\phi_{n+1}^{\prime}(z^{-1})\phi_{n+1}(z)}{z^{n+3}}+o(1)  \Big|^2,
\end{align}
where we have used that $z^{-(n+1)}\phi_{n+1}(z^{-1})o(z^n)=o(1)$ in the last equality.

By the first limit in \eqref{phizero2} and the limit \eqref{phiasy1}, we have
\begin{equation}\label{tgtnum1}
\lim_{n\rightarrow \infty}-\frac{\phi_{n+1}(z^{-1})\phi_{n+1}(z)\frac{d}{dz}D(z^{-1})}{z^{n+1}D(z^{-1})}=0.
\end{equation}
From the second limit in \eqref{phizero2} and again \eqref{phiasy1}, it follows  that
\begin{equation}\label{tgtnum2}
\lim_{n\rightarrow \infty}\frac{\phi_{n+1}^{\prime}(z^{-1})\phi_{n+1}(z)}{z^{n+3}}=0.
\end{equation}
Combining \eqref{tgtdm}, \eqref{tgtnum}, \eqref{tgtnum1}, and \eqref{tgtnum2}, shows
$$\lim_{n\rightarrow \infty}\frac{\left|\phi_{n+1}^{*}(z)\phi_{n+1}^{\prime}(z)-\phi_{n+1}^{*\ \prime}(z)\phi_{n+1}(z)  \right|^2}{\left(\left|\phi_{n+1}^{*}(z)\right|^2-\left|\phi_{n+1}(z)\right|^2\right)^2}=0.$$

Therefore using the above limit and Theorem \ref{OPUC} for the intensity function $h_n^P(z)$, we see that
\begin{align*}
\lim_{n\rightarrow \infty} h_n^P(z)&=\lim_{n\rightarrow \infty}\frac{1}{\pi\left(1-|z|^2 \right)^2}\left(1-\frac{\left(1-|z|^2\right)^2\left|\phi_{n+1}^{*}(z)\phi_{n+1}^{\prime}(z)-\phi_{n+1}^{*\ \prime}(z)\phi_{n+1}(z)  \right|^2}{\left(\left|\phi_{n+1}^{*}(z)\right|^2-\left|\phi_{n+1}(z)\right|^2\right)^2}   \right) \\
&=\frac{1}{\pi\left(1-|z|^2 \right)^2},
\end{align*}
which gives the limiting case for $|z|>1$.

Combining the limiting cases when $|z|<1$ and $|z|>1$, we have now established result of Theorem \ref{OPUC1}.
\end{proof}

\textbf{Acknowledgements.}  The author would like to thank his advisor Igor Pritsker for all his help with the project and for helping with financial support through his grant from the National Security Agency.  The author would also like to thank Jeanne LeCaine Agnew Endowed Fellowship and the Vaughn Foundation via Anthony Kable for financial support.

\end{document}